\documentclass[12pt,reqno,a4paper]{amsart}
\usepackage{a4wide}
\usepackage{amssymb}
\usepackage{graphicx}
\usepackage{subfig}
\usepackage{dsfont}
\usepackage{color}
\usepackage{psfrag}
\usepackage{enumitem}
\usepackage{hyperref}
\usepackage{wrapfig}
\usepackage{verbatim}
\usepackage{diagbox}

\numberwithin{equation}{section}

\advance \vsize by 1cm
\renewcommand{\epsilon}{\varepsilon}
\newcommand{\Rr}{\mathbb{R}}
\newcommand{\Nn}{\mathbb{N}}

\DeclareMathOperator{\dom}{dom}

\theoremstyle{plain}
\newtheorem{theo}{Theorem}[section]
\newtheorem{lem}[theo]{Lemma}
\newtheorem{prop}[theo]{Proposition}
\newtheorem{coro}[theo]{Corollary}

\theoremstyle{definition}
\newtheorem{rem}[theo]{Remark}

\newtheorem{example}[theo]{Example}

\begin{document}

\author[M. Pierre]{Morgan Pierre}
\title[Maximum time step for the BDF3 scheme]{Maximum time step for the BDF3 scheme applied to gradient flows}

\begin{abstract}
For backward differentiation formulae (BDF) applied to gradient flows of semiconvex functions, qua\-dra\-tic sta\-bi\-li\-ty implies the existence of a Lyapunov functional. 
We compute the maximum time step which can be derived from  quadratic stability for the 3-step BDF method (BDF3).  Applications to the asymptotic behaviour of sequences generated by the BDF3 scheme are given. 
\end{abstract}

\address{Morgan Pierre, Laboratoire de Math\'ematiques et Applications, Universit\'e de Poi\-tiers, CNRS, F-86073  Poitiers, France.}
\email{morgan.pierre@math.univ-poitiers.fr}

\maketitle
\noindent
{\bf Keywords:}  gradient system, BDF method, semiconvex function, Kur\-dy\-ka--{\L}o\-ja\-sie\-wicz property, multivalued dynamical system.


\section{Introduction}
In this paper, we focus on the 3-step backward differentiation formula (BDF3) applied to the gradient flow of a semiconvex function in finite dimension. It is known that if the time step is small enough, the BDF3 scheme  is a gradient system, which means that  it is possible to  find a Lyapunov function for the discrete-in-time dynamical system~\cite{SH}. 
A fundamental consequence is that  the time discrete model mimics the asymptotic behaviour of the gradient flow.

The construction of the Lyapunov function for the BDF3 scheme involves quadratic forms. A similar  construction is also well-known for the BDF1 and BDF2 methods~\cite{SH}. It was successfully generalized  into the notion of ``quadratic stability'' for BDFk schemes in~\cite{BPA}. In particular, the BDF schemes of order 4 and 5 were  proved to be gradient systems (or ``gradient stable''). BDF methods of order $k\ge 7$ are not zero-stable~\cite{HNW}, so they cannot be gradient systems, but it is still not known whether the BDF6 scheme is a gradient system or not.

For BDFk methods ($1\le k\le 5$) applied to the gradient flow of a convex function, gradient stability holds without any restriction on the time step. However, for a semiconvex function, it is easily seen that a restriction on the time step is required.

Since quadratic stability implies gradient stability, it is interesting  to compute the maximum time step which can be obtained from quadratic stability. This is easily obtained for the BDF1 and BDF2 schemes~\cite{BPA}. In this paper, we compute the maximum time step for the BDF3 scheme. Surprisingly, this  allows the dynamical system associated to the BDF3 scheme to be multivalued {\it and} gradient stable if the time step is near the optimal value.

Using the gradient stability, we are then able to prove that a bounded sequence generated by the BDF3 scheme converges to a single equilibrium for a large class of functions. For this purpose, we apply a general result on descent methods due to Attouch, Bolte and Svaiter~\cite{ABS} (see also~\cite{AP,Betal}), and we assume a Kurdyka-{\L}ojasiewicz type condition on the Lyapunov function associated to the scheme. Our regularity assumptions on the nonlinearity are weaker than in the references~\cite{BPA,SH}: our framework includes also non-differentiable functions.

\smallskip


Our paper is organized as follows.  We first compute in Section~\ref{secbeta3} the maximum value for the time step of the BDF3 scheme based on quadratic stability. Then, in Section~\ref{secBDF3}, we prove the gradient stability of the BDF3 scheme for this maximal value and we derive some consequences on the asymptotic behaviour of sequences generated by the BDF3 scheme.

\section{Optimal constant for the quadratic stability of the BDF3 scheme}
\label{secbeta3}
We consider the following quadratic form on $\Rr^3$, which will be related to the BDF3 scheme in Section~\ref{secBDFkqs}:
\begin{equation}
\label{g3bis}
\gamma_3(x_1,x_2,x_3)=\frac{11}{6}x_1^2-\frac{7}{6}x_1x_2+\frac{1}{3}x_1x_3.
\end{equation}
We have 
\begin{prop}
\label{propqs3}
The BDF3 scheme is quadratically stable, namely there exist a positive definite quadratic form $q_3$ on $\Rr^2$ and a positive definite quadratic form $r_3$ on $\Rr^3$ such that 
\begin{equation}
\label{qs3}
\gamma_3(x_1,x_2,x_3)=q_3(x_1,x_2)-q_3(x_2,x_3)+r_3(x_1,x_2,x_3),
\end{equation}
 for all $(x_1,x_2,x_3)\in\Rr^3$.
\end{prop}
\begin{proof}
It is easy to check (see~\cite[p. 423]{SH} or~\cite[Formula (17)]{BPA}) that~\eqref{qs3} holds with
$q_3(x_1,x_2)=\cfrac{5}{12}x_1^2+\cfrac{1}{6}(x_1-x_2)^2$
and 
\begin{equation}
\label{qs3bis}
r_3(x_1,x_2,x_3)=\frac{5}{6}x_1^2+\frac{1}{4}(x_1-x_2)^2+\frac{1}{6}(x_1-x_2+x_3)^2.
\end{equation}
\end{proof}

We note that the quadratic forms $q_3$ and $r_3$ in~\eqref{qs3} are {\it not} uniquely defined.  
Indeed, if $q_3$ and $r_3$ are positive definite forms which satisfy~\eqref{qs3}, then for a quadratic form $q_3^\epsilon$ close enough to $q_3$, $q_3^\epsilon$ is positive definite  and the quadratic form $r_3^\epsilon$ defined by~\eqref{qs3} is also positive definite (use for instance Sylvester's criterion).

We define $\beta_3$ as the supremum of real numbers $\beta>0$ such that 
\begin{equation}
\label{defbeta3}
\gamma_3(x_1,x_2,x_3)=q_3(x_1,x_2)-q_3(x_2,x_3)+\tilde{r}_3(x_1,x_2,x_3)+\beta x_1^2
\end{equation}
 where  $q_3$ is a positive definite quadratic form on $\Rr^2$ and $\tilde{r}_3$ is a positive definite quadratic form on $\Rr^3$. Formula~\eqref{qs3bis} shows that $\beta_3\ge 5/6$. The remainder of this section is devoted to the proof of the following
\begin{theo}
\label{theobeta3}
We have $\beta_3=95/96$ and 
$$\gamma_3(x_1,x_2,x_3)=q_3^\star(x_1,x_2)-q_3^\star(x_2,x_3)+\tilde{r}_3^\star(x_1,x_2,x_3)+\beta_3 x_1^2\mbox{ with }$$ 
$q_3^\star(x_1,x_2)=\cfrac{1}{6}(x_2-\cfrac{7}{4}x_1)^2+\cfrac{1}{6}x_1^2$ and $\tilde{r}_3^\star(x_1,x_2,x_3)=\cfrac{1}{6}(x_3-\cfrac{7}{4}x_2+x_1)^2$.
\end{theo}
We note that $\tilde{r}_3^\star$ is positive semidefinite, but not positive definite.  
\begin{proof}Let $q_3$ be a positive definite quadratic form on $\Rr^2$. Then $(x_1,x_2)\mapsto q_3(x_2,x_1)$ is also positive definite, and using its Cholesky decomposition, we obtain that  
$$q_3(x_1,x_2)=a^2x_2^2+2acx_2x_1+(b^2+c^2)x_1^2$$ 
for some unique real numbers $a>0$, $b>0$ and $c\in\Rr$. Thus, 
$$r_3(x_1,x_2,x_3)=\gamma_3(x_1,x_2,x_3)-q_3(x_1,x_2)+q_3(x_2,x_3)$$
 reads 
\begin{eqnarray*}
r_3(x_1,x_2,x_3)&=&\frac{11}{6}x_1^2-\frac{7}{6}x_1x_2+\frac{1}{3}x_1x_3-(ax_2^2+2acx_2x_1+(b^2+c^2)x_1^2)\\
&&+(ax_3^2+2acx_3x_2+(b^2+c^2)x_2^2).
\end{eqnarray*}
Next, we perform a Gauss reduction  of $r_3$. We obtain
\begin{eqnarray}
r_3(x_1,x_2,x_3)&=&(ax_3+cx_2+\frac{1}{6a}x_1)^2\nonumber\\
&&+(b^2-a^2)\left(x_2-\frac{1}{2(b^2-a^2)}\left(\frac{7}{6}+2ac+\frac{c}{3a}\right)x_1\right)^2\nonumber\\
&&+f(a,b,c)x_1^2,\label{r3abc}
\end{eqnarray}
where 
\begin{equation}
\label{deff}
f(a,b,c)=\frac{11}{6}-\frac{1}{4(b^2-a^2)}\left(\frac{7}{6}+2ac+\frac{c}{3a}\right)^2-\left(b^2+c^2+\frac{1}{36a^2}\right).
\end{equation}
Thus, $r_3$ is positive definite if and only if $a>0$, $b^2-a^2>0$ and $f(a,b,c)>0$. Moreover, it is clear from~\eqref{r3abc} that $\beta_3$ is the supremum $f(a,b,c)$ over the set of real numbers such that $a>0$,  $b>a$ and $c\in\Rr$.  In Lemma~\ref{lemg3}, we show that this supremum is equal to $95/96$. 
\end{proof}

\begin{lem}
\label{lemg3}
Let $\Omega=\{(a,b,c)\in\Rr^3\ :\ a>0\mbox{ and }b>a\}$. Then 
$$\beta_3=\sup_\Omega f=\cfrac{95}{96}.$$
 Moreover, for any sequence $(a_n,b_n,c_n)$ in $\Omega$ such that $f(a_n,b_n,c_n)\to \beta_3$, we have 
$(a_n,b_n,c_n)\to (\cfrac{1}{\sqrt{6}},\cfrac{1}{\sqrt{6}},-\cfrac{7}{4\sqrt{6}})$.
\end{lem}
\begin{proof}We define $\overline{f}:\overline{\Omega}\to\Rr\cup\{-\infty\}$  as the lowest upper semicontinuous function above $f$ on the closure $\overline{\Omega}$ of $\Omega$. Namely, for each $(a,b,c)\in\overline{\Omega}$, 
$$\overline{f}(a,b,c)=\sup\left\{\limsup_{n\to+\infty} f(a_n,b_n,c_n) \ |\ (a_n,b_n,c_n)\in\Omega,\ (a_n,b_n,c_n)\to (a,b,c)\right\}.$$
Then $\overline{f}$ is upper semicontinuous on $\overline{\Omega}$ and $\sup_{\overline{\Omega}}\overline{f}=\sup_\Omega f$  (see, e.g.,~\cite[Section 1.1.1]{GMSII}).

Now, let $(a_n,b_n,c_n)$ be a sequence in $\Omega$ such that $f(a_n,b_n,c_n)\to \beta_3=\sup_\Omega f$. Since $\beta_3>0$ by Proposition~\ref{propqs3}, for $n$ large enough we have $f(a_n,b_n,c_n)>0$ and since $b_n>a_n$, the definition~\eqref{deff} of $f$  yields $b_n^2+c_n^2<11/6$ and $11/6>1/(36a_n^2)$. In particular, $(a_n)$, $(b_n)$ and $(c_n)$ are bounded and $a_n\ge 1/\sqrt{66}$.   Thus, up to a subsequence, $(a_n,b_n,c_n)$ converges in $\overline{\Omega}$ to a point $(a^\star,b^\star,c^\star)$ such that $b^\star\ge a^\star\ge 1/\sqrt{66}>0$. Since $\overline{f}=f$ on $\Omega$, the definition of $\overline{f}$ yields 
$$\overline{f}(a^\star,b^\star,c^\star)=\sup_\Omega f=\beta_3.$$
Since $a^\star>0$, then $(a^\star,b^\star,c^\star)$ either belongs to $\Omega$ or to its boundary with $b^\star=a^\star>0$. 

We first assume that $(a^\star,b^\star,c^\star)\in\Omega$. Then $\nabla f(a^\star,b^\star,c^\star)=(0,0,0)$. A calculation (with \texttt{Maple}) yields $\nabla f$. We first use that $\cfrac{\partial f}{\partial b}(a^\star,b^\star,c^\star)=0$. Since $\cfrac{\partial f}{\partial b}(a,b,c)=2b\cfrac{\delta^2}{\eta^2}-2b$
with 
$$\delta=\left(\frac{7}{6}+2ac+\frac{c}{3a}\right)\quad\mbox{ and }\quad\eta=2(b^2-a^2),$$
 this yields $\delta^\star=\epsilon\eta^\star$ with $\epsilon\in\{-1,1\}$, at the critical point. Next, we use $\cfrac{\partial f}{\partial c}(a^\star,b^\star,c^\star)=0$ with 
$\cfrac{\partial f}{\partial c}(a,b,c)=-\cfrac{2\delta(a+1/(6a))}{\eta}-2c$. This yields  
\begin{equation}
\label{cstar}
c^\star=-\epsilon \left(a^\star+\frac{1}{6a^\star}\right).
\end{equation}
We plug this into $\delta^\star$ and we use $\eta^\star=\epsilon\delta^\star$. We obtain
$$\eta^\star=\epsilon\frac{7}{6}-2a^{\star 2}-\frac{2}{3}-\frac{1}{18a^{\star 2}}.$$
We must have $\eta^\star>0$, because $b^\star>a^\star$. For $\epsilon=-1$, this is not possible. For $\epsilon=+1$, we find that it is also not possible. Indeed, in the latter case, we have
$$\eta^\star=\frac{1}{2}-2a^{\star 2}-\frac{1}{18a^{\star 2}}=-\frac{1}{18a^{\star 2}}(-9a^{\star 2}+36a^{\star 4}+1),$$
and the discriminant of the equation is $\Delta=9^2-4\times 36<0$, so this quantity $\eta^\star$ is (strictly) negative for all values of $a^\star$.

Thus, the point $(a^\star,b^\star,c^\star)$ necessarily belongs to the boundary of $\Omega$, with $b^\star=a^\star>0$. We compute $\overline{f}$ at $(a^\star,a^\star,c^\star)$. Since $b_n\to a^\star$ with $b_n>a_n$, from the expression~\eqref{deff} of $f$ we see that 
\begin{equation}
\label{defcbis}
\frac{7}{6}+2a^\star c^\star+\frac{c^\star}{3a^\star}=0,
\end{equation}
otherwise we would have $f(a_n,b_n,c_n)\to-\infty$, a contradiction. We note that in the right hand-side of~\eqref{deff}, the second term is nonpositive. 
Thus,  the value $\bar{f}(a^\star,a^\star,c^\star)$, which is  a supremum of $f$, is best reached  by a sequence  which tends to $(a^\star,a^\star,c^\star)$ and satisfies the constraint~\eqref{defcbis}. 
We obtain that 
$$\bar{f}(a^\star,a^\star,c^\star)=\frac{11}{6}-\left(a^{\star 2}+c^{\star 2}+\frac{1}{36a^{\star 2}}\right),$$
with the constraint~\eqref{defcbis}. It remains to solve this constrained optimization problem. From~\eqref{defcbis} we find the value of $c^\star=-7/(6(2a^\star+1/3a^\star))$  in terms of $a^\star$. We introduce the function 
$$g(a)=\frac{11}{6}-a^2-\frac{49}{36\left(2a+\frac{1}{3a}\right)^2}-\frac{1}{36a^2},$$
so that $\overline{f}(a^\star,a^\star,c^\star)=g(a^\star)$. We seek the maximum of $g$ on $(0,+\infty)$. A calculation  yields
$$g'(a)=-\frac{(6a^2-1)(36a^4+33a^2+1)(36a^4-9a^2+1)}{18(6a^2+1)^3a^3}.$$
The only positive root of $g'$ is $a^\star=1/\sqrt{6}$. From the variations of $g$, we see that $a^\star$ is the unique maximum of $g$ on $(0,+\infty)$.  A computation yields $g(a^\star)=95/96$. The uniqueness of the maximizer of $\bar{f}$ implies the convergence of the whole sequence. This concludes the proof. 
\end{proof}

\section{Gradient stability of the BDF3 scheme}
\label{secBDF3}
\subsection{Assumptions} Let $F:\Rr^M\to\Rr\cup\{+\infty\}$ be a proper lower semicontinuous function. We assume that $F$ is {\it semiconvex}, i.e. there exists $c_F\ge 0$ such that  
\begin{equation}
\label{Fsc}
\mbox{the function }\tilde{F}:V\mapsto F(V)+\frac{c_F}{2}\|V\|^2\mbox{ is convex.}
\end{equation}
Here, $\|\cdot\|$ is the euclidean norm and  $\langle\cdot,\cdot\rangle$ will be the scalar product in $\Rr^M$. 

It is easily seen that if~\eqref{Fsc} holds for some constant $c_F\ge 0$, there exists a minimal value $c_F^\star\ge 0$ for which~\eqref{Fsc} is true, and we denote $c_F$ this optimal value. In particular, if $F$ is convex, we have $c_F=0$. 

The domain of $F$ is the convex set 
$$\dom(F)=\{V\in\Rr^M\ :\ F(V)<+\infty\}(\not=\emptyset)$$
 and we assume that 
\begin{equation}
\label{domF}
F\mbox{ is continuous on }\dom(F).
\end{equation}

Finally, we  assume that 
\begin{equation}
\label{Finf}
\inf_{\Rr^M}F>-\infty.
\end{equation}

Since $\tilde{F}$ is convex, we may define its subdifferential~\cite{B,RW} as a set-valued map $\partial\tilde{F}:\Rr^M\rightrightarrows\Rr^M$.
The subdifferential of $F$ is defined accordingly as the set-valued map $\partial F:\Rr^M\rightrightarrows\Rr^M$ through $\partial F(V)=\partial\tilde{F}(V)-c_FV$, that is
\begin{equation}
\label{defdF}
W\in\partial F(V)\iff\forall V'\in\Rr^M,\  F(V')\ge F(V)+\langle W,V'-V\rangle-\frac{c_F}{2}\|V'-V\|^2.
\end{equation}
We note that $\dom(\partial F)\subset \dom(F)$, where 
$$\dom(\partial F)=\{V\in \Rr^M\ :\ \partial F(V)\not=\emptyset\}.$$

\subsection{BDF methods for gradient flows}
First consider the gradient flow 
\begin{equation}
\label{GF}
U'(t)\in-\partial F(U(t)),\quad t\ge 0,
\end{equation}
in $\Rr^M$, where $U:[0,+\infty)\to\Rr^M$ is continuous.  Since $\partial\tilde{F}$ (cf.~\eqref{Fsc}) is a maximal monotone operator in $\Rr^M$, for every $U_0\in \dom(F)$, there exists a unique strong solution $U$ to~\eqref{GF} such that $U(0)=U_0$~\cite{B}. Moreover, $t\mapsto F(U(t))$ is nonincreasing and several consequences on the asymptotic behaviour of $U(t)$ can be derived~\cite{B,HJ}.  

Here, we focus on the time discretization of~\eqref{GF} by backward differentiation formulae (BDF). 
Let $\Delta t>0$ be the time step. The general $k$-step  BDF scheme for~\eqref{GF} is defined  by 
\begin{equation}
\label{BDFk}
\sum_{j=1}^k\frac{1}{j}\partial^jU^{n+k}\in-\Delta t\partial F(U^{n+k}),\quad n\ge 0,
\end{equation}
where, for a sequence $(U^n)_{n\ge 0}$, the backward difference operator $\partial^j$ is defined recursively by $\partial^jU^{n}=\partial^{j-1}(U^n-U^{n-1})$ ($j\ge 2$, $n\ge j$). When $j=1$, we have $\partial U^n=U^n-U^{n-1}$. 

The one-step BDF method is the backward Euler scheme:
\begin{equation}
\label{BDF1}
U^{n+1}-U^n\in-\Delta t\partial F(U^{n+1}),\quad n\ge 0.
\end{equation}
Any solution to the proximal algorithm solves the BDF1 scheme, but the converse is true only if $\Delta t$ is small enough (Proposition~\ref{propexist}). 
The two-step BDF method reads
\begin{equation}
\label{BDF2}
\frac{3}{2}U^{n+2}-2U^{n+1}+\frac{1}{2}U^n\in-\Delta t\partial F(U^{n+2}),\quad n\ge 0,  
\end{equation}
and the three-step BDF method reads
\begin{equation}
\label{BDF3}
\frac{11}{6}U^{n+3}-3 U^{n+2}+\frac{3}{2} U^{n+1}-\frac{1}{3} U^n\in-\Delta t\partial F(U^{n+3}),\quad n\ge 0.
\end{equation}

 If $F\in C^{k+2}(\Rr^M,\Rr)$ ($k\le 6$) and if the initial conditions are well chosen,  the error between the solution $U$ of~\eqref{GF}  and its approximation given by the BDF scheme~\eqref{BDFk} is of order $O(\Delta t^k)$ on finite time intervals~\cite[Theorem 3.5.7]{SH}.

\medskip

Let $k$ be a positive integer. 
We denote
\begin{equation}
\label{alphak}
\alpha_k=\sum_{j=1}^k1/j>0. 
\end{equation}
The assumptions on $F$ imply~\cite{BPA}:
\begin{prop}
\label{propexist}
Let $U^0$, \ldots $U^{k-1}$ be given in $\Rr^M$. For all $\Delta t>0$,  there exists a least one sequence $(U^n)_{n\ge 0}$ in $\Rr^M$ with  initial values $U^0$, \ldots $U^{k-1}$ which complies with~\eqref{BDFk}. If $c_F\Delta t<\alpha_k$, this sequence is unique. 
\end{prop}
\begin{rem}
\label{rem1BDFk}
If $c_F\Delta t\ge \alpha_k$, then there may be more than one solution to the BDFk scheme, for a given set of initial values (cf. example below).
\end{rem}
\begin{example}
  Let $F:\Rr\to\Rr$ be a function of class $C^\infty$ such that $F''(v)=-\alpha_k$ on $[-1,1]$, $F''(v)\ge -\alpha_k$ on $\Rr$ and $F(v)\to +\infty$ as $|v|\to +\infty$, then $c_F=\alpha_k$ and for $\Delta t=1$, the BDFk scheme~\eqref{BDFk} reads
  $$F'(u_{n+k})+\alpha_ku_{n+k}-l_n=0,\quad n\ge 0,$$
 where $l_n=l_n(u_{n+k-1},\ldots,u_n)$. If $l_n=F'(0)$ then $u_{n+k}$ may take any value in $[-1,1]$.  
\end{example}

\subsection{Quadratic stability of the BDF3 method}
\label{secBDFkqs}
Following the approach in~\cite{BPA,SH}, we multiply the left-hand side of~\eqref{BDFk} by $\partial U^{n+k}$, and we consider the quantity
\begin{equation}
\label{Gk}
\Gamma_k=\sum_{j=1}^k\frac{1}{j}\langle\partial^jU^{n+k},\partial U^{n+k}\rangle.
\end{equation}
For $k=3$, we consider $\Gamma_3$ as a quadratic form depending on the variables $(\partial U^{n+3},\partial U^{n+1},\partial U^{n+1})$. Namely, 
\begin{eqnarray}
\Gamma_3&=&\frac{11}{6} \|\partial U^{n+3}\|^2-\frac{7}{6}\langle\partial U^{n+3},\partial U^{n+2}\rangle+\frac{1}{3}\langle\partial U^{n+3},\partial U^{n+1}\rangle.\label{G3aux}
\end{eqnarray} 

The results from the previous section can be used for $\Gamma_3$. More precisely, if 
$$q(x_1,\ldots,x_d)=\sum_{i=1}^d\sum_{j=1}^da_{ij}x_ix_j$$
 is a quadratic form on $\Rr^d$, associated to the symmetric matrix $A=(a_{ij})_{1\le i,j\le d}$, we define the following quadratic form on $(\Rr^M)^d$:
$$Q(V_1,\ldots,V_d)=\sum_{i=1}^d\sum_{j=1}^da_{ij}\langle V_i,V_j\rangle.$$
Then $Q$ inherits the properties of $q$. In particular, if $q$ is positive definite, then so is $Q$~\cite[Section 3.1]{BPA}. 

By comparing~\eqref{G3aux} to~\eqref{g3bis}, we see that the quadratic form in $(\Rr^M)^3$ associated to $\gamma_3$ is precisely $\Gamma_3$.  Theorem~\ref{theobeta3} shows that for each $\beta<\beta_3=95/96$, there exist positive definite quadratic forms $q_3^\beta$ on $\Rr^2$ and $\tilde{r}_3^\beta$ on $\Rr^3$ such that
$$\gamma_3(x_1,x_2,x_3)=q_3^\beta(x_1,x_2)-q_3^\beta(x_2,x_3)+\tilde{r}_3^\beta(x_1,x_2,x_3)+\beta x_1^2.$$ 
We denote $Q_3^\beta$ and $\tilde{R}_3^\beta$ the corresponding quadratic forms on $(\Rr^M)^2$ and $(\Rr^M)^3$. There are positive definite and we have
\begin{equation}
\label{G3qs}
\Gamma_3(V_1,V_2,V_3)=Q_3^\beta(V_1,V_2)-Q_3^\beta(V_2,V_3)+\tilde{R}_3^\beta(V_1,V_2,V_3)+\beta\|V_1\|^2.
\end{equation}
We note that $Q_3^\beta$ and $\tilde{R}_3^\beta$ depend on $\beta$.

\subsection{Gradient stability of the BDF3 scheme}
We use the positive semidefinite quadratic forms $Q_3^\beta$ and $\tilde{R}_3^\beta$ defined in~\eqref{G3qs}. 
 For $\hat{V}=(V_0,V_1,V_2)\in(\Rr^M)^3$, we define
\begin{equation}
\label{defhF}
\hat{F}_3^\beta(\hat{V})=F(V_0)+\frac{1}{\Delta t}Q_3^\beta(V_1,V_2).
\end{equation}
For a sequence $(U^n)_{n\ge 0}$ in $\Rr^M$, we denote
$$\hat{U}^{n+3}=(U^{n+3},\partial U^{n+3},\partial U^{n+2}),$$
so that 
\begin{equation}
\label{defhFbis}
\hat{F}_3^\beta(\hat{U}^{n+3})=F(U^{n+3})+\frac{1}{\Delta t}Q_3^\beta(\partial U^{n+3},\partial U^{n+2}).
\end{equation}

 By setting $k=3$ in~\eqref{BDFk}, the BDF3 methods reads: for each $n\in\Nn$,  
\begin{equation}
\label{aux3}
\exists W^{n+3}\in\partial F(U^{n+3})\mbox{ such that } W^{n+3}= -\frac{1}{\Delta t}\sum_{j=1}^3\frac{1}{j}\partial^jU^{n+3}.
\end{equation}
The following result shows the gradient stability of the BDF3 scheme.  
\begin{theo}
\label{theoBDF3}
Let $(U^n)$ be a sequence in $\Rr^M$ which complies with the BDF3 scheme~\eqref{BDF3}. Assume that  $c_F\Delta t<2\beta_3$ and  let  $\beta$ be chosen arbitrarily in $[c_F\Delta t/2,\beta_3)$.  Then for each $n\ge 0$ we have
\begin{equation}
\label{aux5}
\hat{F}_3^\beta(\hat{U}^{n+3})+\frac{1}{\Delta t}\tilde{R}_3^\beta(\partial U^{n+3},\partial U^{n+2},\partial U^{n+1})\le \hat{F}_3^\beta(\hat{U}^{n+2}).
\end{equation}
\end{theo}
\begin{proof} Let $n\ge 0$ and  $W\in \partial F(U^{n+3})$. We apply the definition~\eqref{defdF} of $\partial F$ with $V=U^{n+3}$ and $V'=U^{n+2}$. This yields
\begin{equation}
\label{aux1}
F(U^{n+2})\ge F(U^{n+3})-\langle W,\partial U^{n+3}\rangle-\cfrac{c_F}{2}\|\partial U^{n+3}\|^2.
\end{equation}
We choose $W=W^{n+3}$ from the BDF3 scheme~\eqref{aux3} in this estimate and we  use the definition~\eqref{Gk} of $\Gamma_3$.  We obtain
$$F(U^{n+2})\ge F(U^{n+3})+\frac{1}{\Delta t}\Gamma_3(\partial U^{n+3},\partial U^{n+2},\partial U^{n+1})-\cfrac{c_F}{2}\|\partial U^{n+3}\|^2.$$
From~\eqref{G3qs}, ~\eqref{defhFbis} and $\beta\ge c_F\Delta t/2$, we deduce~\eqref{aux5}. 
\end{proof}
If $(U^n)_n$ is a sequence in $\Rr^M$, we denote 
$$\omega((U^n)_n)=\left\{U^\star\in\Rr^M\ :\ \exists n_p\to +\infty\mbox{ such that }U^{n_p}\to U^\star\right\}$$
its $\omega$-limit set. The set of critical points of $F$ is 
$$\mathcal{S}=\left\{V\in \Rr^M\ :\ 0\in\partial F(V)\right\}.$$

As a consequence of gradient stability, we have:
\begin{coro}
\label{coroBDF3}
Let $(U^n)$ be a bounded sequence in $\Rr^M$ which complies with the BDF3 scheme~\eqref{BDF3}. If $c_F\Delta t<2\beta_3$, then $\partial U^n\to 0$ and $\omega((U^n)_n)$ is a compact and connected subset of $\Rr^M$ included in $\mathcal{S}$.
\end{coro}
\begin{proof}
Let $\beta\in [c_F\Delta t/2,\beta_3)$.  
By~\eqref{aux5}, the sequence $(\hat{F}_3^\beta(\hat{U}^{n+3}))_n$ is nonincreasing.  It is also bounded from below thanks  to the positivity of $Q_3^\beta $ and assumption~\eqref{Finf}.
By induction, we obtain from~\eqref{aux5} that for all $N\ge 1$, 
\begin{equation}
\label{aux2}
\hat{F}_3^\beta(\hat{U}^{n+3})+\frac{1}{\Delta t}\sum_{n=1}^N\tilde{R}_3^\beta(\partial U^{n+3},\partial U^{n+2},\partial U^{n+1})\le \hat{F}_3^\beta(\hat{U}^3).
\end{equation}
Letting $N\to +\infty$ yields
$$\frac{1}{\Delta t}\sum_{n=1}^{+\infty}\tilde{R}_3^\beta(\partial U^{n+3},\partial U^{n+2},\partial U^{n+1})\le \hat{F}_3^\beta(\hat{U}^3)-\inf F<+\infty.$$
We note here that $\hat{F}_3^\beta(\hat{U}^3)<+\infty$ since $U^3\in \dom(\partial F)\subset \dom(F)$. Since $\tilde{R}_3^\beta$ is positive definite, we have 
$$ (\partial U^{n+3},\partial U^{n+2},\partial U^{n+1})\to (0,0,0).$$
Therefore, the bounded sequence $(U^n)$ satisfies $U^{n+1}-U^n\to 0$ and a standard argument shows that $\omega((U^n)_n)$ is a compact and connected subset of $\Rr^M$. 

Let $U^\star=\lim_{p\to+\infty} U^{n_p}$ be an element of $\omega((U^n)_n)$. Then from~\eqref{aux2} we deduce that $\hat{F}_3^\beta(\hat{U}^{n_p+3})\le \hat{F}_3^\beta(\hat{U}^3)$ and since $F$ is lower semicontinuous, we see that 
$$F(U^\star)\le\liminf_p\hat{F}_3^\beta(\hat{U}^{n_p+3})<+\infty. $$
Thus, $U^\star\in \dom(F)$. 
  Now, since $\partial U^{n+1}\to 0$,  we have $W^{n+3}\to 0$ in~\eqref{aux3}.  
	By letting $n=n_p$ tend to $+\infty$ in~\eqref{aux3}, we obtain $0\in\partial F(U^\star)$. This can easily be seen by choosing $V=U^{n_p+3}$ and $W=W^{n_p+3}$ in~\eqref{defdF}, for an arbitrary $V'\in\Rr^M$, and by letting $p$ tend to $+\infty$. The continuity of $F$ on its domain  yields $F(U^{n_p+3})\to F(U^\star)$.  
\end{proof}
\begin{rem}
\label{remmulti3}
Since $2\beta_3=95/48>\alpha_3=11/6$, if $c_F\Delta t\in [\alpha_3,2\beta_3)$, there may be several  sequences $(U^n)$ which comply with the BDF3 scheme for the same choice of initial conditions. Each one of these sequences is gradient stable.  
The previous known estimate $2\beta_3\ge 5/3 $ did not allow this non-uniqueness phenomenon~\cite{BPA}.  
\end{rem}
\begin{rem}
\label{remcoer}
Assume that $F$ is coercive, i.e. $\lim_{\|V\|\to +\infty}F(V)=+\infty$.
If $(U^n)$ is a sequence in $\Rr^M$ which complies with the BDF3 scheme~\eqref{BDF3} with $c_F\Delta t<2\beta_3$, then by~\eqref{aux2}, the sequence $(F(U^n))_n$ is bounded, so $(U^n)_n$ is bounded as well. 
\end{rem}
In general, the $\omega$-limit set in Corollary~\ref{coroBDF3} is not  reduced to a single point (see~\cite{AMA} for related counter-examples).  However, there are many situations where this happens. In the next result, we use the definition of the {\it Kurdyka-{\L}ojasiewicz (KL) property} as given, e.g., in~\cite{ABS,BCL}. 
 To the class of KL functions belong real analytic, semi-algebraic, real sub-analytic, uniformly convex and convex functions satisfying a growth assumption~\cite{ABRS,BDL,BDLS,BDLM,L}. 
\begin{coro}
\label{coroKL}
Assume that the hypotheses of Corollary~\ref{coroBDF3} are satisfied and let $\beta\in[c_F\Delta t/2,\beta_3)$.  If the function $\hat{F}_3^\beta:(\Rr^M)^3\to \Rr$ has the KL-property at some point $(U^\star,0,0)$ in $(\Rr^M)^3$ where $U^\star\in\omega((U^n)_n)$, then the whole sequence $(U^n)_n$ converges to $U^\star$. 
\end{coro}
\begin{proof}We apply~\cite[Theorem 2.9]{ABS} to the sequence $(\hat{U}^{n+3})_n$ in $(\Rr^M)^3$.  The function $\hat{F}_3^\beta$ is  proper and lower semicontinuous on $(\Rr^M)^3$. It is also semiconvex with constant $c_F$.  We only need to check assumptions  {\bf H1},  {\bf H2} and  {\bf H3} in~\cite{ABS}.
Estimate~\eqref{aux5} shows that {\bf H1} is satisfied.  Corollary~\ref{coroBDF3} shows that for some subsequence, we have $\hat{U}^{n_p+3}\to (U^\star,0,0)$ and 
$$\hat{F}_3^\beta(\hat{U}^{n_p+3})\to \hat{F}_3^\beta(U^\star,0,0)=F(U^\star),$$
 so that {\bf H3} is also satisfied. 
Next, we turn to {\bf H2}. By definition, the positive definite quadratic form $Q_3^\beta$  reads 
$$Q_3^\beta(V_1,V_2)=a\|V_1\|^2+2b\langle V_1,V_2\rangle+c\|V_2\|^2,$$
with $a>0$ and $ac-b^2>0$. Thus,
$$\frac{\partial Q_3^\beta}{\partial V_1}(V_1,V_2)=2aV_1+2bV_2\mbox{ and }\frac{\partial Q_3^\beta}{\partial V_2}(V_1,V_2)=2bV_1+2cV_2.$$
For each $n$, the vector 
$$\hat{W}^{n+3}=(W^{n+3},2a\partial U^{n+3}+2b\partial U^{n+2},2b\partial U^{n+3}+2c\partial U^{n+2})$$
where $W^{n+3}$ solves~\eqref{aux3} belongs to $\partial F_3^\beta(\hat{U}^{n+3})$. Moreover, using~\eqref{aux3} again, we see that  
$$\|\hat{W}^{n+3}\|\le c_1\|(\partial U^{n+3},\partial U^{n+2},\partial U^{n+1})\|\le c_2\|\hat{U}^{n+3}-\hat{U}^{n+2}\|,$$
for some positive constants  $c_1$, $c_2$ independent of $n$. This proves {\bf H2} and concludes the proof. 
\end{proof}
\begin{example}The convergence result of Corollary~\ref{coroKL} can be applied  to the time and space discretization of the Allen-Cahn equation with polynomial nonlinearity considered in~\cite[Section 6]{BPA}. In this case, the function $F$ is polynomial  on $\Rr^M$, so $\hat{F}_3^\beta$ is polynomial as well. Thus, it satisfies the classical {\L}ojasiewciz inequality~\cite{L}.  
\end{example}
\subsection{A barrier to gradient stability}
\label{seclambdak}
We consider a counter-example to gradient stability in the one-dimensional case. 
For  $k\in\{1,2,3\}$, we define $\lambda_k>0$ such that the sequence $u_n=(-1)^n$ solves the BDFk scheme
$$\sum_{j=1}^k\frac{1}{j}\partial^ju^{n+k} =-\lambda_k u^{n+k}.$$
Using~\eqref{BDF1}-\eqref{BDF3}, we see that this holds for $\lambda_1=2$, $\lambda_2=4$ and $\lambda_3=20/3$. 

Let now $\Delta t>0$ and  $ c_{F_k}=\lambda_k/\Delta t$.
It is easy to build a function $F_k$ of class $C^\infty$ on $\Rr$ such that $F_k'(v)=-c_{F_k}v$ on $[-1,1]$, $F_k''(v)\ge -c_{F_k}$ on $\Rr$  and $F_k(v)\to +\infty$ as $|v|\to +\infty$.   This function satisfies assumptions~\eqref{Fsc}-\eqref{Finf}, and by construction the sequence $u_n=(-1)^n$ is bounded and complies with the BDFk scheme~\eqref{BDFk} for $F_k$. However, for this time step, the BDFk scheme does not satisfy the conclusions of Corollary~\ref{coroBDF3}  because  $\pm 1$ are not critical points of $F_k$. Thus, the number $\lambda_k$ appears as a ``barrier'' to the gradient stability of the BDFk scheme.

\smallskip

The results are summarized in Table~\ref{table1}. The values $\beta_1=\beta_2=1$ are easy to find~\cite[Section 3.4]{BPA}. For $k=1$, we have $\lambda_1=2\beta_1=2$, so that quadratic stability gives the optimal time step for gradient stability. For $k\ge 2$, we have $\lambda_k>2\beta_k$, so there is  a gap between quadratic stability and gradient stability.
It could be interesting to understand the asymptotic behaviour of sequences when $c_F\Delta t$ belongs to the interval $[2\beta_k,\lambda_k)$.
  
No restriction on the time step is required if the BDF1 scheme is replaced by the proximal algorithm. Moreover, for the proximal algorithm, the semiconvexity assumption on $F$ can be removed~\cite{AB,MP}. 

\begin{table}[ht] 
\begin{tabular}{|c||c|c|c|}
\hline
$k$&$1$&$2$&$3$\\
\hline
$\alpha_k$&$1$&$3/2$&$11/6$\\
\hline
$2\beta_k$&$2$&$2$&$95/48$\\
\hline
$\lambda_k$&$2$&$4$&$20/3$\\
\hline
\end{tabular}
\caption{\label{table1}Uniqueness and stability numbers for BDFk methods}
\end{table}

\section*{Acknowledgements}
The author is thankful to Fr\'ed\'eric Bosio, Anass Bouchriti and Nour Eddine Alaa for helpful discussions.


\end{document}